\documentclass[dvips,preprint]{article}
\usepackage[latin1]{inputenc}
\usepackage{amsmath, amssymb, amsthm}
\usepackage[dvips]{graphicx}
\usepackage{mathrsfs}
\usepackage{showlabels}
\numberwithin{equation}{section}

\newcommand{\bcov}{\operatorname{\mathbb{C}ov}}
\newcommand{\bvar}{\operatorname{\mathbb{V}ar}}

\newcommand{\Id}{\operatorname{Id}}

\newcommand{\ind}[1]{\mathbf{1}_{#1}}
\newtheorem{theo}{Theorem}[section]
\newtheorem{pr}{Proposition}[section]
\newtheorem{cor}{Corollary}[section]
\newtheorem{lem}{Lemma}[section]

\newtheorem{rmk}{Remark}[section]

\begin{document}

\title{Ruelle-Perron-Frobenius operator approach to the annealed pinning model with Gaussian long-range correlated disorder}
\author{Julien Poisat \footnotemark[1]}

\footnotetext[1]{Mathematical Institute, Leiden University, P.O. Box 9512, 2300 RA Leiden, The Netherlands.\\
E-mail: poisatj@math.leidenuniv.nl\\
This work was supported by the french ANR project MEMEMO2 10--BLAN--0125--03.\\
}
\maketitle

\begin{abstract}
In this paper we study the pinning model with correlated Gaussian disorder. The presence of correlations makes the annealed model more involved than the usual homogeneous model, which is fully solvable. We prove however that if the disorder correlations decay fast enough then the annealed critical behaviour is the same as the homogeneous one. Our result is sharper if the decay is exponential. The approach we propose relies on the spectral properties of a transfer or Ruelle-Perron Frobenius operator related to the model. We use results on these operators  that were obtained in the framework of the thermodynamic formalism for countable Markov shifts. We also provide large-temperature asymptotics of the annealed critical curve under weaker assumptions.
\end{abstract}

\vspace{1.5cm}
The random pinning model applies to various situations such as localization of a polymer on a defect line, wetting transition and DNA denaturation, which all display a transition between a localized phase and a delocalized phase. In the last past years considerable progress has been made on the understanding of the critical phenomena in the case of an environment constituted of independent and identically distributed (i.i.d) random variables \cite{Alexander_Sido, Alexander_quenched, Alexander_loop_exp_one, Cheliotis_DenH, MR2779401, MR2231963, MR2561435, Lacoin_martingale, Fabio_replica, 1157.60090}, and the correlated case is now under investigation \cite{BergerLacoin, 2011arXiv1110.5781B, poisat_preprint_relevance, 0903.3704v3}. Here we study the critical properties of the annealed (ie averaged over disorder) system when disorder is a Gaussian correlated sequence, and we prove in particular that if the decay of correlations is fast enough then the annealed free energy at criticality behaves as in the homogeneous (i.e nonrandom) model. To this end, we show a connection between our model and the study of a transfer operator on infinite sequences. We then use results on its spectral property coming from the theory of thermodynamic formalism for countable Markov shifts \cite{MR1955261, MR1738951, MR2249785, MR1853808}, generalizing thereby the approach of \cite{0903.3704v3} that was used to deal with correlations with finite range.

This paper is organized as follows: in a first part we recall the general model and standard definitions, and we give a brief account on what is known in the literature about critical features of the model. In Section 2, we specify the model, compute the annealed partition functions for correlated Gaussian disorder and state our main results. Section 3 contains the definitions and theorems on countable Markov shifts that we use in Section 4 to prove Theorems \ref{critical_sum} and \ref{critical_exp}. Section 5 contains the proof of Theorem \ref{asympt} and can be read independently of Sections 3 and 4.

\section{Introduction}

\paragraph{The general model.}
We first present the random pinning model in a general setting. The reader can also refer to \cite{GG_Book, Toninelli_Survey} and \cite[Chapters 7 and 11]{DenH_Book}. Let $(T_n)_{n\geq1}$ be a sequence of i.i.d random variables in $\mathbb{N}^*$ (the interarrival times), with law denoted by $P$. We suppose that 
\begin{equation}\label{defK}
 K(n) := P(T_1 = n) = L(n)n^{-(1+\alpha)} >0
\end{equation}
with $\alpha\geq0$ and $L$ a slowly varying function (see \cite{MR1015093}). We denote by $\tau$ the renewal process defined by $\tau_0 = 0$ and $\tau_{n}-\tau_{n-1} = T_n$ for $n\geq1$, and define the random variables
\begin{align*}
\delta_n &:= \ind{\{n\in\tau\}} = \left\{ \begin{array}{cl} 1 & \mbox{if $\tau_k = n$ for some $k\geq0$}\\0 & \mbox{otherwise} \end{array} \right. ,\\
\imath_n &:= \sum_{k=1}^n \delta_k.
\end{align*}
We suppose that $\tau$ is recurrent in the sense that $\sum_{n\geq1}K(n)=1$. Let $\omega = (\omega_n)_{n\geq0}$ be a sequence of real random variables, independent of $\tau$. Its law is denoted by $\mathbb{P}$. The pinning probability measure in quenched environnment $\omega$, at size $n\geq1$ and parameters $\beta\geq0$ (the inverse temperature) and $h\in\mathbb{R}$ (the pinning parameter) is defined by
\begin{equation}\label{new_law}
 \frac{dP_{n,\beta,h}^{\omega}}{dP} := \frac{1}{Z_{n,\beta,h}^{\omega}}\exp\left( H_{n,\beta,h}^{\omega}  \right) \delta_n
\end{equation}
where
\begin{equation*}
 H_{n,\beta,h}^{\omega} := \sum_{k=1}^n (\beta\omega_k + h)\delta_k
\end{equation*}
is the Hamiltonian function and 
\begin{equation}\label{part_fct}
 Z_{n,\beta,h}^{\omega} := E\left( \exp\left( H_{n,\beta,h}^{\omega} \right)\delta_n\right)
\end{equation}
is the quenched partition function. Define the finite volume quenched free energies by
\begin{equation*}
\forall n\geq 1, \quad F^{\omega}_{n,\beta,h} := (1/n)\log Z_{n,\beta,h}^{\omega}.
\end{equation*}
The following result is now standard (see \cite[Theorem 4.6]{GG_Book}) and can be proved with Kingman's ergodic subadditivity theorem:
\begin{theo}
 If $\omega$ is stationary ergodic and $\mathbb{E}(|\omega_0|)<+\infty$, then the sequence $(F_{n,\beta,h}^{\omega})_{n\geq1}$ converges $\mathbb{P}-$a.s and in $L^1(\mathbb{P})$ to a nonnegative and nonrandom quantity called (infinite volume) quenched free energy and denoted by $F(\beta,h)$. Moreover, $F(\beta,h) = \sup_{n\geq1} \mathbb{E}F^{\omega}_{n,\beta,h}$.
\end{theo}
\begin{rmk}\label{pinned_free}
 The constraint $\delta_n$ appearing at the right of Equations (\ref{new_law}) and (\ref{part_fct}) says that we look at realizations of $\tau$ such that $n\in\tau$. In this case, one speaks of \textit{pinned} or \textit{constraint} partition functions. This constraint can be safely removed (in which case one speaks of \textit{free} partition functions) without modification of the infinite volume free energy, see \cite[Remark 1.2]{GG_Book}.
\end{rmk}
Note that by suitably shifting $h$ one can assume that $\omega_0$ is centered without losing in generality. The localized and delocalized phases of the system are respectively defined by
\begin{align*}
 \mathscr{L} &= \{(\beta,h)\in \mathbb{R}_+ \times \mathbb{R} : F(\beta,h) > 0\},\\
  \mathscr{D} &= \{(\beta,h)\in \mathbb{R}_+ \times \mathbb{R} : F(\beta,h) = 0\}.
\end{align*}
For a fixed $\beta$, the quenched free energy is nondecreasing (and convex) in $h$. Therefore, a phase transition occurs at the quenched critical point (when it is finite)
\begin{equation*}
 h_c(\beta) := \sup\{ h : F(\beta,h) = 0\}.
\end{equation*}
The quenched critical exponent, which governs the behaviour of the quenched free energy at the neighbourhood of $h_c(\beta)$, is another important critical feature of the model.

\begin{rmk}
 If one wants to see the delocalization transition as a denaturation transition for DNA, then the $(T_n)_{n\geq1}$ would (more or less) stand for the lengths of denaturation loops in a double-stranded DNA molecule and $(\omega_n)_{n\geq0}$ for the nucleotides sequence. The presence of long-range correlations in these sequences makes the study of the correlated case relevant. See \cite[Section 1.4]{GG_Book} for a more detailed description of the relation between the pinning model and Poland-Scheraga models.
\end{rmk}

\paragraph{Homogeneous case.} Let us now briefly recall the critical features of the homogeneous model ($\beta=0$), which is fully solvable. One can prove that $h_c(0)=0$ and (see \cite[Theorem 2.1]{GG_Book} for a more precise statement)
\begin{theo}\label{hom_crit_exp}
 For every choice of $\alpha\geq0$ and $L$ in (\ref{defK}), there exists a slowly varying function $\tilde{L}$ such that
 \begin{equation*}
  F(0,\delta) \stackrel{\delta\searrow0}{\sim} \delta^{\max(1,1/\alpha)}\tilde{L}(1/\delta).
 \end{equation*}
Moreover, $\lim_{\delta\searrow}\tilde{L}(1/\delta) = 1/m$ when $m := \sum_{n\geq1} nK(n) < +\infty$.
\end{theo}
Therefore, the critical exponent of the homogeneous pinning model is equal to $\max(1,1/\alpha)$.

\paragraph{I.I.D case and the annealed model.} A lot of work has been achieved recently on the critical phenomenon of pinning models in i.i.d disorder. Here we restrict ourselves to i.i.d standard Gaussian random variables, for ease of exposition and also because in this paper we investigate the case of correlated \textit{Gaussian} random variables. Most of statements in this paragraph can be extended to other laws with finite exponential moments. Let us define the annealed partition functions by
\begin{equation*}
 Z_{n,\beta,h}^a := \mathbb{E}Z_{n,\beta,h}^{\omega}.
\end{equation*}
By applying Jensen's inequality, one gets $\mathbb{E}\log Z_{n,\beta,h}^{\omega} \leq \log Z_{n,\beta,h}^a$, which yields 
\begin{align}
 F(\beta,h) &\leq F^a(\beta,h) := \lim_{n\rightarrow+\infty} (1/n)\log Z_{n,\beta,h}^a,\\
 h_c(\beta) &\geq h_c^a(\beta) := \sup\{ h : F^a(\beta,h) = 0\}.
\end{align}
These two lines are in fact always true as soon as $F^a(\beta,h)$ and $h_c^a(\beta)$, respectively called annealed free energy and annealed critical point, are well defined. This fact is easily checked here, since by direct computation we have
\begin{equation*}
 \forall n\geq 1,\quad Z^a_{n,\beta,h}  =Z_{n,0,h+\beta^2/2},
\end{equation*}
from which we infer
\begin{align}
  F^a(\beta,h) &= F(0,h+\beta^2/2)\label{ann_iid}\\
 h_c^a(\beta) &= -\beta^2/2\label{ann_cc_iid}\\
 h_c(\beta) &\geq -\beta^2/2.
\end{align}
From Equation (\ref{ann_iid}) we deduce that the annealed critical exponent is the same as the homogeneous critical exponent given by Theorem \ref{hom_crit_exp}. Furthermore, a smoothing inequality obtained in \cite{MR2231963} states that for all values of $\alpha$,
\begin{equation}\label{smoothing}
\forall \delta\geq0,\quad F(\beta,h_c(\beta)+\delta)\leq \frac{1+\alpha}{2\beta^2}\delta^2,
\end{equation}
which indicates that quenched and annealed critical exponent cannot coincide if $\alpha> 1/2$. A series of papers then proved the following dichotomy, in accordance with the Harris criterion (see \cite{harris1974effect}, \cite[Section 5.5]{GG_Book} and \cite[Section 5]{Toninelli_Survey} for a more detailed explanation):
\begin{enumerate}
 \item If $0<\alpha<1/2$: there exists $\beta_c \in (0,+\infty)$ such that the annealed and quenched critical points and exponents coincide if $\beta \in (0,\beta_c)$ (irrelevant regime) and $h_c(\beta)> h_c^a(\beta)$ if $\beta>\beta_c$ (relevant regime).
 \item If $\alpha>1/2$: for all $\beta>0$, $h_c(\beta)>h_c^a(\beta)$.
 \end{enumerate}
We let aside the special cases $\alpha=0$ (irrelevance for all $\beta>0$, see \cite{Alexander_loop_exp_one}) and $\alpha=1/2$ (marginal case). The reader can refer to \cite{Alexander_quenched, Alexander_loop_exp_one, Alexander_Sido, Fabio_replica, Lacoin_martingale, Cheliotis_DenH, 1157.60090, MR2561435, MR2779401, MR2231963} for precise results.

\paragraph{Correlated case.}
Some models with correlated disorder have been recently investigated.

In \cite{BergerLacoin}, the renewal process evolves in an environment constituted of independent strips of $0$ and $-\beta$'s, where the distribution for the size of the strips has a power-law tail with exponent $\theta>1$. This particular environment is an example of a disorder potential with long-range power-law decaying correlations. It is then proved that the critical point is $0$ for all $\beta>0$ and that the critical exponent is $\theta/\min(1,\alpha)$ (with explicit logarithmic corrections).

In \cite{2011arXiv1110.5781B}, a hierarchical pinning model with correlated disorder is studied. The choice of correlations would correspond to power-law decaying correlations if one makes the analogy with the non-hierarchical model. The authors identify three regimes: non-summable correlations (the phase transition disappears), fast-decaying correlations (the annealed exponent is the same as in the homogeneous model and the (ir)relevance criterion for disorder is as in the i.i.d case) and a third regime where the annealed critical exponent is modified.

The random pinning model in the case of Gaussian disorder with finite-range correlations has been studied in \cite{poisat_preprint_relevance} and \cite{0903.3704v3}. In this case, the annealed partition function is no more a homogeneous partition function. However, one can write it (approximately) like the homogeneous partition function for a Markov renewal process (depending on $\beta$) with a shift on the pinning paramater, and see the Markov renewal process as a renewal process with a finite range memory (somehow the dependence structure of the disorder sequence is transferred to the renewal process upon annealing). The method that we use to prove Theorems \ref{critical_sum} and \ref{critical_exp} is a natural but not straightforward extension to the infinite-range case of the techniques used for the finite-range case, which rely on Perron-Frobenius tools. Furthermore, in the finite-range case the Harris criterion is not modified.

\section{Model and results}

Our results deal with the annealed critical features of the random pinning model with Gaussian disorder under some tail assumptions on the correlations.

 \paragraph{Assumptions and preliminaries.} In the following $\omega = (\omega_n)_{n\geq0}$ is a stationary sequence of Gaussian random variables satisfying 
\begin{equation*}
 \mathbb{E}(\omega_0) = 0 \quad \mbox{ and } \quad \mathbb{E}(\omega_0^2) = 1.
\end{equation*}
We define
\begin{equation*}
 \forall n\geq1 \quad \rho_n = \bcov(\omega_0,\omega_n).
\end{equation*}
 The first step is to compute the annealed partition function. We have
\begin{equation}\label{ann_part_fct}
 Z_{n,\beta,h}^a = E\left[ \exp\left(\left(h+\frac{\beta^2}{2}\right)\sum_{k=1}^n \delta_k + \beta^2 \sum_{1\leq k<l\leq n}\rho_{k-l}\delta_k \delta_l \right) \delta_n \right].
\end{equation}
Indeed, this a direct consequence of 
\begin{align*}
 \bvar\left(\sum_{k=1}^n \omega_k \delta_k \right) &= \sum_{1\leq k , l \leq n} \rho_{|k-l|} \delta_k \delta_l \\
 &= \sum_{k=1}^n \delta_k + 2 \sum_{1\leq k < l \leq n} \rho_{k-l} \delta_k \delta_l.
\end{align*}
and of the fact that $\sum_{k=1}^n \omega_k \delta_k$ is a Gaussian random variable conditionally on $\tau$. The second term in the exponential in Equation (\ref{ann_part_fct}) makes this model more complicated than homogeneous pinning.

We now prove the existence of the annealed free energy under weak tail assumptions on the correlations.
\begin{pr}\label{existence}
If $\sum |\rho_n| <\infty$ then the annealed free energy 
\begin{equation*}
 F^a(\beta,h) = \lim_{n\rightarrow +\infty} (1/n) \log Z^a_{n,\beta,h}
\end{equation*}
exists and is finite.
\end{pr}

\begin{proof}[Proof of Proposition \ref{existence}]
 We fix $\beta$ and $h$ and write $Z_n^a$ as a shortcut for $Z_{n,\beta,h}^a$. Let us define
 \begin{equation*}
  \Delta_n = \sum_{k=1}^n \sum_{i\geq k} |\rho_i| = \sum_{k=1}^n k|\rho_k|.
 \end{equation*}
By restricting the partition function to paths with $\delta_n = 1$ and using Markov property we get
\begin{equation*}
 Z^a_{n+m} \geq Z^a_n Z^a_m \exp\left( -\beta^2 \Delta_{m} \right)\geq Z^a_n Z^a_m \exp\left( -\beta^2 \Delta_{n+m} \right).
\end{equation*}
Since $\sum |\rho_n| < \infty$,
\begin{equation}
\frac{\Delta_n}{n(n+1)} = \frac{\Delta_n}{n} - \frac{\Delta_{n+1}}{n+1} + |\rho_{n+1}|
\end{equation}
is summable. We can conclude that the annealed free energy exists by applying the approximate subadditive lemma (see Lemma \ref{lemmaHammersley} below). Again, it is finite since
\begin{equation}
Z_n^a \leq \exp\left( n \left( h+ \frac{\beta^2}{2} + \beta^2 \sum_{k\geq1} |\rho_k| \right) \right).
\end{equation}
\end{proof}

\begin{lem}[Hammersley's approximate subadditive lemma, \cite{Hammersley}]\label{lemmaHammersley}
 Let $h : \mathbb{N} \mapsto \mathbb{R}$ such that for all $n$, $m\geq1$,
 \begin{equation*}
  h(n+m) \leq h(n) + h(m) + \Delta(n+m),
 \end{equation*}
where $(\Delta_n)_{n\geq0}$ is a non-decreasing sequence such that:
\begin{equation*}
 \sum_{r\geq 1} \frac{\Delta(r)}{r(r+1)} < +\infty.
\end{equation*}
Then the sequence $(h(n)/n)_{n\geq1}$ has a limit in $[-\infty,+\infty)$.
\end{lem}

\paragraph{Main results} 

We now state our main results. Theorem \ref{critical_sum} states that if the correlations decay fast enough then the annealed critical behaviour is the same as in the homogeneous case (see Theorem \ref{hom_crit_exp}). Under the stronger assumption that the correlations decay exponentially fast, a sharper version of the critical behaviour can be derived, as stated in Theorem \ref{critical_exp}.

\begin{theo}\label{critical_sum}
 If $\sum_{n\geq1} n|\rho_n| < \infty$ then for all $\beta>0$, in all of the following cases there exist a constant $c_{\beta}\in(0,1)$ and a slowly varying function $\tilde{L}$ such that
\begin{enumerate}
 \item if $K(n) = L(n) n^{-(1+\alpha)}$, $0<\alpha < 1$, for $\delta$ positive and small enough,
 \begin{equation*}
  c_{\beta} \tilde{L}(1/\delta)\delta^{1/\alpha} \leq F^a(\beta,h_c^a(\beta)+\delta) \leq (1/c_{\beta}) \tilde{L}(1/\delta)\delta^{1/\alpha},
 \end{equation*}
\item if $K(n) = L(n)/n^2$ with $\sum_{n\geq1}L(n)/n = +\infty$, for $\delta$ positive and small enough,
 \begin{equation*}
  c_{\beta} \tilde{L}(1/\delta)\delta \leq F^a(\beta,h_c^a(\beta)+\delta) \leq (1/c_{\beta}) \tilde{L}(1/\delta)\delta,
 \end{equation*}
 where $L(1/\delta) \stackrel{\delta\rightarrow 0}{\rightarrow} 0$,
 \item if $\sum_{n\geq1} n K(n) < +\infty$, then for all $\delta>0$
 \begin{equation}\label{critical_sum_moy}
    c_{\beta}\delta \leq F^a(\beta,h_c^a(\beta)+\delta) \leq \delta.
 \end{equation}
\end{enumerate}
\end{theo}

\begin{theo}\label{critical_exp}
 Under the following assumption:
 \begin{equation*}
\exists C>0,\,\varrho\in(0,1) : \forall n\geq 1\quad  |\rho_n| \leq C \varrho^n,
 \end{equation*}
for all $\beta>0$ and in all of the following cases there exist a constant $c_{\beta}>0$ and a slowly varying function $\tilde{L}$ such that
\begin{enumerate}
 \item if $K(n) = L(n) n^{-(1+\alpha)}$, $0<\alpha < 1$, 
 \begin{equation*}
  F^a(\beta,h_c^a(\beta)+\delta) \stackrel{\delta\searrow 0}{\sim} c_{\beta} \tilde{L}(1/\delta)\delta^{1/\alpha},
 \end{equation*}
\item if $K(n) = L(n)/n^2$ with $\sum_{n\geq1}L(n)/n = +\infty$,
 \begin{equation*}
  F^a(\beta,h_c^a(\beta)+\delta) \stackrel{\delta\searrow 0}{\sim} c_{\beta} \tilde{L}(1/\delta)\delta,
 \end{equation*}
 where $\tilde{L}(1/\delta) \stackrel{\delta\rightarrow 0}{\rightarrow} 0$,
 \item if $\sum_{n\geq1} n K(n) < +\infty$, 
 \begin{equation*}
    F^a(\beta,h_c^a(\beta)+\delta) \stackrel{\delta\searrow 0}{\sim} c_{\beta} \delta.
 \end{equation*}
\end{enumerate}
\end{theo}

\begin{rmk}
 It appears in the proofs that the slowly varying functions $\tilde{L}$ of Theorems \ref{critical_sum} and \ref{critical_exp} are the same as those of Theorem \ref{hom_crit_exp}.
\end{rmk}

Along the proofs of Theorems \ref{critical_sum} and \ref{critical_exp}, we derive an expression of the annealed critical curve in function of the largest eigenvalue of some operator that we define later (see Corollary \ref{cor}). Even if it is not possible to compute it in general, one can give the following large-temperature asymptotics (compare with Equation (\ref{ann_cc_iid})):

\begin{theo}\label{asympt}
Under the assumption of Proposition \ref{existence} we have
 \begin{equation*}
  h_c^a(\beta) \stackrel{\beta\searrow 0}{\sim} -\frac{\beta^2}{2}\left( 1 + 2 \sum_{n\geq 1} \rho_n P(n\in\tau) \right).
 \end{equation*}
\end{theo}

\begin{rmk}
The results of Theorem \ref{critical_sum} were also obtained in \cite[Theorem 5.2.2]{Berger_these} with a different method which does not rely on any knowledge on the annealed critical curve. There it is also proved that the smoothing inequality (\ref{smoothing}) still holds (with a different constant) if $\sum_{n\geq1}|\rho_n|$ is finite, which, combined to the results of Theorem \ref{critical_sum}, means that disorder is relevant if $\alpha>1/2$ and $\sum n|\rho_n|<+\infty$. Reference \cite[Chapter 5]{Berger_these} also contains a discussion on the Weinrib-Halperin criterion applied to the random pinning model with disorder correlations. The Weinrib-Halperin criterion is a generalization of the Harris criterion to correlated disorder (see \cite{MR729982} and \cite{PhysRevB.27.413} on this topic).
\end{rmk}

\paragraph{Outline of the proof of Theorem \ref{asympt}.} The large-temperature asymptotic of the annealed critical curve is not obtained from the characterization of Corollary \ref{cor} but by a truncation argument: we apply Proposition 5.2 of  \cite{0903.3704v3} (which is Theorem \ref{asympt} for finite-range correlations) to a truncated version of the correlation sequence $(\rho_n)_{n\geq0}$ and then make the range of the truncation go to infinity.

\paragraph{Outline of the proofs of Theorems \ref{critical_sum} and \ref{critical_exp}.} A first part is common to both proofs. In the finite-range correlations case of \cite{poisat_preprint_relevance} and \cite{0903.3704v3}, the annealed model has been tackled using a transfer matrix approach and Perron-Frobenius tools. The method also involves a Markov renewal process, which can be seen as a renewal process where one interarrival time depends on the $q$ previous ones, where $q$ is the range of correlations. If one wants to extend these methods to the infinite range case, it is natural to expect the following analogy:\\
\\
\begin{tabular}{|p{5cm}||p{6cm}|}\hline
\multicolumn{1}{|c||}{\textbf{Finite range}} & \multicolumn{1}{c|}{\textbf{Infinite range}} \\
\hline
Transfer matrix on $\mathbb{N^*}^q$ & Transfer operator on functions on $\mathbb{N^*}^{\mathbb{N}}$\\
\hline
Process with finite memory & Process with infinite memory\\
\hline
Perron-Frobenius eigenvalue and eigenvectors & Perron-Frobenius eigenvalue, eigenfunctions , eigenmeasures ?\\
\hline
\end{tabular}\\
\\
Of course, there is no infinite sequence of interarrival times in finite size systems, so we introduced a modified version of the finite-volume annealed partition function by adding a ``past'', i.e an infinite sequence of interarrival times before $\tau_0 = 0$. Fast decay of correlations imply that this operation does not affect the value of the infinite volume annealed free energy. The modified annealed partition function can then be written as iterations of a transfer operator (or Ruelle-Perron-Frobenius operators) acting on functions defined on the space of sequences of integers. This step makes a link between our model and a topic in ergodic theory and dynamical systems called thermodynamic formalism for countable Markov shifts (see Section \ref{CMS}). Some of the main interests there are the existence and properties of Gibbs and equilibrium measures for sequences taking values in a countable alphabet, equipped with the left-shift action, and in the presence of a potential. In this context, a generalized Perron-Frobenius theorem (see Theorem \ref{Ruelle_op_spec} and \cite{MR1955261}) was proved for transfer operators under some regularity of the potential function. It also appears that this regularity assumption translates in our case as a condition on the decay of correlations, namely $\sum n|\rho_n|  < +\infty$. With ``Perron-Frobenius'' eigenfunctions in hands, one can first give a characterization of the annealed critical curve and free energy, and then prove Theorem \ref{critical_sum}. Going to the sharper results of Theorem \ref{critical_exp} requires a non-trivial result on phase transitions for countable Markov shifts (see Theorem \ref{Gurevich_pressure_critical_exp} and \cite{MR2249785}), the potential regularity assumption of which is equivalent to exponential decay of correlations in our case.

\section{Countable Markov shifts}\label{CMS}

In a first part we define some standard objects of the thermodynamic forma\-lism of countable Markov shifts, among which the Ruelle-Perron-Frobenius (or transfer) operators and Gurevich pressure (an analogue of the Perron-Frobenius eigenvalue of finite irreducible matrices). The second part contains two theorems that we use to prove Theorems \ref{critical_sum} and \ref{critical_exp}.

\subsection{Definitions}

Let $S$ be a countable set and $\Sigma = S^{\mathbb{N}}$ the set of $S$-valued sequences. Let $T$ be the left-shift on $\Sigma$, that is
\begin{equation*}
 T : \left\{ \begin{array}{ccc}
             \Sigma & \rightarrow & \Sigma\\
             x = (x_i)_{i\geq0} & \mapsto & Tx = (x_{i+1})_{i\geq0}.
            \end{array}
 \right.
\end{equation*}
The metric on $\Sigma$ is defined by $d(x,y) = 2^{-\inf\{k\geq0 : x_k \neq y_k\}}$, with the convention that $\inf \varnothing = +\infty$. We adopt the following notations for cylinder sets:
\begin{equation*}
\forall n\geq 1,\, \forall a=(a_0,\ldots,a_{n-1})\in S^n,\, [a] := \{x\in \Sigma : x_0 = a_0, \ldots, x_{n-1}=a_{n-1}\}.
\end{equation*}
If $x$ is in $\Sigma$ and $s$ in $S$, we denote by $(sx)$, or $sx$ when there is no confusion, the concatenation of $s$ and $x$, that is $(sx)_0 = s$ and $(sx)_n = x_{n-1}$ for all $n\geq 1$. Notice that in the general case, $T$ may be defined on any shift-invariant subset of $\Sigma$, but here we are only interested in the full shift. In this case, $T$ is clearly topologically mixing. Indeed, for all finite cylinder set $C$, there exists $n\geq1$ (choose $n$ equal to the size of the cylinder) such that $T^n(C) = \Sigma.$

\begin{rmk}\label{bip}
In \cite{MR1955261}, the BIP condition -- for big images and preimages condition -- is introduced. Let $\mathbb{A} = (A_{i,j})_{(i,j)\in S\times S}$ be a matrix taking values in $\{0,1\}$ and suppose we only consider the shift acting on the subset $$\Sigma_{\mathbb{A}} := \{x\in S^{\mathbb{N}} : \forall i\geq 0, A_{x_i,x_{i+1}} = 1\}.$$ Then the BIP condition is satisfied iff there exist an integer $N\geq1$ and $b_1,\ldots,b_N$ in $S$ such that for all $a$ in $S$, there exists $i,j$ such that $A_{b_i, a} A_{a, b_j} = 1$. In our case the state space is $\Sigma_{\mathbb{A}}=S^{\mathbb{N}}$, so all the components of $\mathbb{A}$ are equal to $1$ and the BIP condition is clearly satisfied.
\end{rmk}

\paragraph{Potentials.} A potential is any real valued function $\phi$ defined on $\Sigma$. Potential variations are defined by:
\begin{equation}\label{var_pot}
 \forall n\geq1,\quad V_n(\phi) := \sup\{|\phi(x) - \phi(y)| : x_0 = y_0, \ldots, x_{n-1} = y_{n-1}\}.
\end{equation}
A potential is said to have summable variations if 
\begin{equation*}
 \sum_{n\geq 1} V_n(\phi) < +\infty,
\end{equation*}
and is said to have exponentially fast decaying variations if there exist $\varrho\in(0,1)$ and a constant $C>0$ such that for all $n\geq1$,
\begin{equation*}
 V_n(\phi) \leq C \times \varrho^n.
\end{equation*}
In the last case the potential is said to be H\"older continuous, because for all $x$ and $y$ such that $x_0 = y_0$,
\begin{equation*}
 |\phi(x) - \phi(y)| \leq C \varrho^{\inf\{k\geq0 : x_k \neq y_k\}} = C d(x,y)^{\kappa}
\end{equation*}
with $\kappa := -\log \varrho / \log 2$. For all $n\geq 1$, we write
\begin{equation}\label{phi_n}
 \phi_n = \sum_{k=0}^{n-1} \phi \circ T^k,
\end{equation}
with $T^0 := \Id$.

\paragraph{Transfer operators.} The transfer operator (also known as Ruelle-Perron-Frobenius operator) associated to a potential $\phi$, denoted by $\mathcal{L}_{\phi}$, is defined by:
\begin{equation*}
\forall f\in \mathbb{R}^\Sigma,\, \forall x\in \Sigma, \quad (\mathcal{L}_{\phi}f)(x) = \sum_{y: Ty = x} e^{\phi(y)}f(y),
\end{equation*}
which also writes:
\begin{equation}\label{op_def}
 (\mathcal{L}_{\phi}f)(x) = \sum_{s\in S}e^{\phi(sx)}f(sx).
\end{equation}
For the definition to be correct, we need to specify on which domain the operator is well defined, that is when the series in Equation (\ref{op_def}) converges. If for instance $\|\mathcal{L}_{\phi}\mathbf{1} \|_{\infty}$ is finite (here $\mathbf{1}$ is the function identically equal to $1$), then the transfer operator acts on bounded functions. By iterating $\mathcal{L}_{\phi}$ we get
\begin{equation}\label{iter_operator_fct}
(\mathcal{L}_{\phi}^nf)(x) = \sum_{y : T^ny = x} e^{\phi_n(y)}f(y) = \sum_{s_1,\ldots, s_n \in S} e^{\phi_n(s_n\ldots s_1 x)}f(s_n\ldots s_1 x)                                                                                                                                                                                                                                                                    \end{equation}
where $\phi_n$ was defined in Equation (\ref{phi_n}).

\paragraph{Gurevich pressure.} The following proposition is taken from Theorem 1 in \cite{MR1738951}.
\begin{pr}\label{def_gurevich_pressure}
Let $a$ be an element of $S$ and $[a] = \{x \in\Sigma : x_0 = a\}$. Suppose that $\sum_{n\geq1}V_n(\phi)< +\infty$. Then the following limit
\begin{equation}\label{def_gurevich}
 P_G(\phi) := \lim_{n\rightarrow +\infty} \frac{1}{n} \log \sum_{x : T^nx = x} e^{\phi_n(x)}\mathbf{1}_{[a]}(x),
\end{equation}
called Gurevich pressure of $\phi$, exists and does not depend on $a$. Moreover, it is finite if $\|\mathcal{L}_{\phi}\mathbf{1} \|_{\infty}$ is finite.
\end{pr}

\begin{proof}
 The proof can be found in \cite{MR1738951} and relies on subadditivity. Indeed, if we define 
\begin{equation}\label{def_Zbold} 
 \mathbf{Z}_n(\phi,a) := \sum_{x : T^nx = x} e^{\phi_n(x)}\mathbf{1}_{[a]}(x),
\end{equation} 
(not to be confused with the partition functions of the pinning model) then for all positive integers $m$ and $n$,
 \begin{equation*}
  \mathbf{Z}_{n+m}(\phi,a) \geq C \mathbf{Z}_n(\phi,a) \mathbf{Z}_m(\phi,a)
 \end{equation*}
where $C=\exp(-3\sum_{n\geq1}V_n(\phi))$. From standard subadditivity argument we deduce that $$P_G(\phi) = \lim_{n\rightarrow+\infty} (1/n)\log \mathbf{Z}_n(\phi,a)$$ exists and that for all $n$,
\begin{equation}\label{P_G_suradd}
 P_G(\phi) \geq \frac{1}{n}\log \mathbf{Z}_n(\phi,a) + \frac{\log C}{n}.
\end{equation}
Finally, $\mathbf{Z}_n(\phi,a) \leq \|\mathcal{L}_{\phi}\mathbf{1} \|_{\infty}^n$ implies $P_G(\phi)\leq \log \|\mathcal{L}_{\phi}\mathbf{1} \|_{\infty}$.
\end{proof}

\begin{rmk}\label{CN_gurevich}
 If $P_G(\phi)$ is finite then we easily get $P_G(\phi + c)= P_G(\phi) + c$ for every constant $c$.
\end{rmk}

\begin{rmk}\label{rmk_gurevich}
 A sufficient condition for finiteness of $\|\mathcal{L}_{\phi}\mathbf{1} \|_{\infty}$ is:
 \begin{equation*}
  \exists C>0,\, u:S\mapsto \mathbb{R}_+ : \sum_{s\in S}u(s)<+\infty\, \mbox{\rm \,and }\, \forall x\in \Sigma,\, e^{\phi(x)}\leq C u(\pi_0(x)),
 \end{equation*}
 where $\pi_0 : \overline{t} = (t_0,t_1,\ldots) \mapsto t_0$ is the projection onto the first coordinate. 
\end{rmk}

\begin{rmk}\label{alter_pression}
 Here are some useful alternative expressions for the Gurevich pressure of a potential with summable variations. First, if the BIP condition holds, then from \cite[Corollary 1]{MR1955261} one has (compare with Equation (\ref{def_gurevich})):
 \begin{equation*}
  P_G(\phi) = \lim_{n\rightarrow +\infty} \frac{1}{n}\log \sum_{x : T^nx = x}e^{\phi_n(x)}.
 \end{equation*}
Suppose now that $\phi$ is a potential with summable variations. We have for all $n\geq 1$ and all $x\in\Sigma$,
\begin{equation*}
 \forall k_1,\ldots,k_n\geq1,\quad |\phi_n([k_1\ldots k_n]^{\texttt{per}}) - \phi_n(k_1\ldots k_n x)| \leq \sum_{m\geq1}V_m(\phi) < +\infty
\end{equation*}
where $[k_1\ldots k_n]^{\texttt{per}}$ is the periodization of the word $k_1\ldots k_n$. From this we deduce that for all $x\in\Sigma$,
\begin{equation}\label{alter_def}
 P_G(\phi) = \lim_{n\rightarrow +\infty} (1/n)\log (\mathcal{L}^n_{\phi}\mathbf{1})(x).
\end{equation}
It is now straightforward using Equations (\ref{iter_operator_fct}) and (\ref{alter_def}) to show that for every sequence of functions $(f_n)_{n\geq1}$ satisfying $$0 < \inf_{n\geq1, y\in\Sigma} f_n(y) \leq \sup_{n\geq1, y\in\Sigma} f_n(y) < +\infty,$$ we have
\begin{equation}\label{alter_def2}
 \forall x\in\Sigma,\quad P_G(\phi) = \lim_{n\rightarrow +\infty} (1/n)\log (\mathcal{L}^n_{\phi}f_n)(x).
\end{equation}
\end{rmk}

\subsection{Theorems}

Under some strong enough assumptions, the transfer operator $\mathcal{L}_{\phi}$ has some properties similar to those of irreducible positive matrices, and an analogue of the Perron-Frobenius theorem can be proved. Historically this was first established by Ruelle for finite alphabets (and more generally in the compact case). This was later extended by several authors, and now include shifts on countable alphabets in the case of potential variations decaying fast enough. The reader can refer to \cite{MR1853808,MR0234697,MR1955261,MR2249785,MR1738951,MR1818392} for background. The following theorem is taken from \cite[Corollary 2 and Theorem 2]{MR1955261}. We remind that the BIP condition is satisfied in our case (see Remark \ref{bip}).

\begin{theo}\label{Ruelle_op_spec}
 If $\phi$ has summable variations and finite Gurevich pressure then, writing $\lambda := \exp(P_G(\phi))$, there exist a continuous function $h : \Sigma \mapsto (0,+\infty)$ and a finite Borel measure $\nu$ which is positive on finite cylinder sets, as well as ergodic, such that:
 \begin{enumerate}
  \item $\mathcal{L}_{\phi}h = \lambda h$,
  \item $\mathcal{L}^*_{\phi}\nu = \lambda \nu$, where $\mathcal{L}^*_{\phi}$ is the dual operator of $\mathcal{L}_{\phi}$,
  \item $\int_\Sigma h d\nu < +\infty$,
  \item $\nu(\Sigma) < +\infty$,
  \item $0< \inf h \leq \sup h < + \infty$,
  \item if $h$ is such that $\int_\Sigma h d\nu = 1$, then
  \begin{equation*}
   \forall n\geq1, \forall a\in S^n,\quad \lambda^{-l}\mathcal{L}_{\phi}^l\mathbf{1}_{[a]} \stackrel{l\rightarrow +\infty}{\longrightarrow} \nu[a] h,
  \end{equation*}
uniformly on every compact set of $\Sigma$.
 \end{enumerate}
\end{theo}

Let $h$ and $\nu$ be as in Theorem \ref{Ruelle_op_spec} and normalized so that $\int_\Sigma h d\nu = 1$, and let $m$ be a probability measure on $\Sigma$ defined by
\begin{equation}\label{Gibbs_measure}
 dm = h d\nu.
\end{equation}
Then $m$ is a shift-invariant \textit{Gibbs measure} associated to $\phi$, that is there exists a constant $B>0$ so that for all $n\geq1$, $a$ in $S^n$ and $x$ in $[a]$,
\begin{equation*}
 \frac{1}{B} \leq \frac{m([a])}{e^{\phi_n(x) - nP_G(\phi)}} \leq B,
\end{equation*}
see \cite{MR1955261}. From $\nu$ and $h$ we define a Markov chain on $\Sigma$, the transition probabilities of which are:
\begin{equation}\label{chaine_induite}
 \forall x,y\in \Sigma,\quad Q(x,y) := \frac{e^{\phi(y)}h(y)}{\lambda h(x)} \mathbf{1}_{\{Ty = x\}},
\end{equation}
with $\lambda = \exp(P_G(\phi))$. Indeed we have $\sum_{s\in S} Q(x,sx)=1$ for all $x$ in $\Sigma$, from item (1) of Theorem \ref{Ruelle_op_spec}. One can check that $m$ is a stationary distribution for this Markov chain: for all bounded measurable function $\varphi$,
\begin{align*}
 \int (Q\varphi)(x)dm(x)& = \int \sum_{s\in S} \varphi(sx) Q(x,sx) dm(x) \\
 & =\lambda^{-1} \int \sum_{s\in S} \varphi(sx) e^{\phi(sx)} h(sx) d\nu(x)\\
 &= \lambda^{-1} \int \mathcal{L}_{\phi}(\varphi h)(x) d\nu(x)\\
 &= \int \varphi(x)  h(x)d\nu(x) \mbox{\rm \, from item (2) of Theorem \ref{Ruelle_op_spec}}\\
 &= \int \varphi(x) dm(x).
\end{align*}

The second theorem is about the sharp behaviour of the Gurevich pressure for a one-parameter family of potentials $\{\phi + t\psi\}_{t\geq 0}$, when $t$ goes to $0$, under some assumptions on $\phi$ and $\psi$. This result is taken from \cite[Theorems 4 and 5, Remark (3) on page 635]{MR2249785}.

\begin{theo}\label{Gurevich_pressure_critical_exp}
 Let $\phi$ and $\psi$ be potentials with exponentially decreasing variations such that
 \begin{enumerate}
  \item $\sup \phi < + \infty$,
  \item $\sup \psi<+\infty$,
  \item $P_G(\phi)<+\infty$.
 \end{enumerate}
 Let $m$ be the invariant Gibbs measure associated to $\phi$, which is defined by Equation (\ref{Gibbs_measure}). In the following, $\tilde{L}$ is a slowly varying function.
 \begin{enumerate}
  \item If $m(\{x : \psi(x) < -n\}) \stackrel{n\rightarrow +\infty}{\sim} \frac{n^{-\alpha}\tilde{L}(n)}{\Gamma(1-\alpha)}$, with $0<\alpha<1$, then
  \begin{equation*}
   P_G(\phi + t\psi) - P_G(\phi) \stackrel{t\searrow 0}{\sim} - \tilde{L}(1/t)t^{\alpha};
  \end{equation*}
\item if $\psi \in L^1(m)$ then
\begin{equation*}
 P_G(\phi + t\psi) - P_G(\phi)  \stackrel{t\searrow 0}{\sim} t \times \left(\int \psi dm\right);
\end{equation*}
\item if $\alpha=1$ and $\psi\notin L^1(m)$ then
\begin{equation*}
 P_G(\phi + t\psi) - P_G(\phi)  \stackrel{t\searrow 0}{\sim} t \tilde{L}(1/t),
\end{equation*}
where $\tilde{L}(n) \stackrel{n\rightarrow +\infty}{\sim} \int(\psi \vee (-n))dm$.
 \end{enumerate}
\end{theo}

Note that in \cite[Theorem 5]{MR2249785}, the term ``equilibrium measure'' is used instead. Nevertheless, we can safely replace it by ``Gibbs measure'' (ibid. footnote on page 637).

\section{Proofs of Theorems \MakeLowercase{\ref{critical_exp}} and \MakeLowercase{\ref{critical_sum}}}\label{preuvedestheoremes...}

Both proofs use the tools of the previous section in the case $S = \mathbb{N}^*$.

\paragraph{Potential G and transfer operators.} Let us start with some definitions and notations. Let $\Sigma = (\mathbb{N}^*)^{\mathbb{N}}$ be the space of sequences of positive integers. If $\overline{t} = (t_0, t_1, \ldots)$ is in $\Sigma$ and $n\geq 1$, we denote by $(n \overline{t})$ the concatenation $(n,t_0,t_1,\ldots)$. We introduce the potentials
\begin{align}
G(\overline{t}) &= \sum_{k\geq0} \rho_{t_0 + \ldots + t_k},\label{def_potentialG}\\
G^{(n)}(\overline{t}) &= \sum_{k=0}^{n-1} \rho_{t_0 + \ldots + t_k},\quad \forall n\geq1.\label{truncated_G1}
\end{align}
Notice that $G$ is well defined if $\sum |\rho_n|$ is finite and that the truncated versions $G^{(n)}$ are different from the one we use to prove Theorem \ref{asympt} (see Equation (\ref{truncated_G})). 

For all $\beta>0$, define:
\begin{equation*}
 \phi_{\beta} : \left\{ \begin{array}{ccc} \Sigma & \rightarrow & \mathbb{R} \\ \overline{t} & \mapsto & \beta^2 G(\overline{t}) + \log K(t_0) \end{array}\right.,
\end{equation*}
or equivalently,
\begin{equation*}
 \phi_{\beta} := \beta^2 G + \log K \circ \pi_0,
\end{equation*}
($\pi_0$ has been defined in Remark \ref{rmk_gurevich}). If $P_G(\phi_{\beta})$ exists, then it must be finite from Remark \ref{CN_gurevich} ($G$ is bounded). We can then define for all $\beta>0$ and $F\geq 0$,
\begin{equation*}
 \phi_{\beta,F} : \left\{ \begin{array}{ccc} \Sigma & \rightarrow & \mathbb{R} \\ \overline{t} & \mapsto &  \beta^2 G(\overline{t}) + \log K(t_0) - P_G(\phi_{\beta}) - F\times t_0 \end{array}\right.,
\end{equation*}
which also writes
\begin{equation*}
 \phi_{\beta,F} = \beta^2 G + \log K \circ \pi_0 - P_G(\phi_{\beta}) - F \times \pi_0 = \phi_{\beta} - P_G(\phi_{\beta}) - F\times \pi_0.
\end{equation*}
Again, if $P_G(\phi_{\beta,F})$ exists then it is finite from the same argument as above. The transfer operator associated to $\phi_{\beta}$ (resp. $\phi_{\beta,F}$) is denoted by $\mathcal{L}_{\beta}$ (resp. $\mathcal{L}_{\beta,F}$). Therefore we have:
\begin{align*}
 \mathcal{L}_{\beta}f &: \left\{ \begin{array}{ccc} \Sigma & \rightarrow & \mathbb{R} \\ \overline{t} & \mapsto &  \sum_{n\geq1} e^{\beta^2 G(n\overline{t})}K(n) f(n\overline{t}) \end{array}  \right.,\\
 \mathcal{L}_{\beta,F}f &: \left\{ \begin{array}{ccc} \Sigma & \rightarrow & \mathbb{R} \\ \overline{t} & \mapsto &  \sum_{n\geq1} e^{\beta^2 G(n\overline{t})-P_G(\phi_{\beta})-Fn}K(n) f(n\overline{t}) \end{array}  \right. .
\end{align*}
Note that $\mathcal{L}_{\beta}$ is not the same as $\mathcal{L}_{\beta,F=0}$. Let us now remark that the decay of variations of $\phi_{\beta}$ and $\phi_{\beta,F}$ is linked to the decay of the correlations $(\rho_n)_{n\geq0}$. More precisely, we have:
\begin{pr}\label{lien_variations}
 If $\sum_{n\geq 1} n |\rho_n|$ is finite then for all $\beta\geq0$ and $F\geq0$, the potentials $\phi_{\beta}$ and $\phi_{\beta,F}$ have summable variations. If the $\rho_n$'s decrease exponentially fast then so do the variations of $\phi_{\beta}$ and $\phi_{\beta,F}$ (with the same exponent).
\end{pr}
\begin{proof}
 First, we have for all $n\geq 1$,
 \begin{equation}\label{decayG1}
  V_n(\phi_{\beta}) = V_n(\phi_{\beta,F}) = \beta^2 V_n(G).
 \end{equation}
Constants and functions depending only on the first coordinate play no role in the potential variations (see Equation (\ref{var_pot})). Observe now that for all $n\geq 1$ and all $x$, $y$ in $\Sigma$ such that $x_0 = y_0,\ldots, x_{n-1}=y_{n-1}$,
\begin{equation*}
 |G(x) - G(y)| \leq \left|\sum_{k\geq n} \rho_{x_0 + \ldots + x_k} - \sum_{k\geq n} \rho_{y_0 + \ldots + y_k}\right| \leq 2 \sum_{k\geq n}|\rho_k|.
 \end{equation*}
 Therefore,
\begin{equation}\label{decayG2} 
 V_n(G)\leq 2 \sum_{k\geq n}|\rho_k|.
 \end{equation}
 Proposition \ref{lien_variations} is then easily deduced from Equations (\ref{decayG1}) and (\ref{decayG2}).
\end{proof}

\begin{rmk}\label{rmk_existence}
 From Propositions \ref{def_gurevich_pressure} and \ref{lien_variations}, if $\sum_{n\geq 1} n |\rho_n|$ is finite then the limits $P_G(\phi_{\beta})$ and $P_G(\phi_{\beta,F})$ exist and are finite.
\end{rmk}

\begin{pr}\label{fonction_P_G}
 Suppose that $\sum_{n\geq1}n|\rho_n| < + \infty$. For all $\beta>0$, the function $F\in\mathbb{R}_+ \mapsto P_G(\phi_{\beta,F})$ is continuous and (strictly) decreasing. It is equal to $0$ if $F=0$ and goes to $-\infty$ as $F$ goes to $+\infty$.
\end{pr}
\begin{proof}
 First, $P_G(\phi_{\beta})$ and $P_G(\phi_{\beta,F})$ exist and are finite from Remark \ref{rmk_existence}. Moreover (see Remark \ref{rmk_gurevich}),
\begin{equation}\label{null_pressure}
 P_G(\phi_{\beta,F=0}) = P_G(\phi_{\beta}- P_G(\phi_{\beta})) = 0.
 \end{equation}
Let $a$ be any positive integer. For all $F_1$ and $F_2$ in $\mathbb{R}_+$ we have
\begin{equation*}
 \forall \overline{t}\in\Sigma, \quad \phi_{\beta,F_1+F_2}(\overline{t}) \leq \phi_{\beta,F_1}(\overline{t}) - F_2,
\end{equation*}
hence (recall Equation (\ref{def_Zbold}))
\begin{equation*}
 \mathbf{Z}_n(\phi_{\beta,F_1+F_2},a) \leq \mathbf{Z}_n(\phi_{\beta,F_1},a) \exp(-F_2n),
\end{equation*}
and taking the limit in $n$,
\begin{equation}\label{decr_P_G}
 P_G(\phi_{\beta,F_1+F_2}) \leq P_G(\phi_{\beta,F_1}) - F_2. 
\end{equation}
From this we deduce that $F\mapsto P_G(\phi_{\beta,F})$ is strictly decreasing and that $P_G(\phi_{\beta,F})$ tends to $-\infty$ as $F$ tends to $+\infty$. Continuity on $(0,+\infty)$ is a consequence of convexity. Indeed, for all $n$, the function $F\mapsto \log \mathbf{Z}_n(\phi_{\beta,F},a)$ is convex (a property that is conserved by taking the limit in $n$) since for $F>0$, $\partial^2_F \log \mathbf{Z}_n(\phi_{\beta,F},a)$  can be written as a variance, which is nonnegative.
 Let us now prove the continuity (on the right) at $F=0$. Let $\epsilon>0$. From Equation (\ref{null_pressure}),
$$(1/n)\log \mathbf{Z}_n(\phi_{\beta,F=0},a) \stackrel{n\rightarrow +\infty}{\longrightarrow} 0,$$
therefore there exists an integer $n_0$ such that $$\frac{1}{n_0}\log \mathbf{Z}_{n_0}(\phi_{\beta,F=0},a) + \frac{\log C}{n_0} \geq -\epsilon,$$
where $C$ is the constant appearing in Equation (\ref{P_G_suradd}). From the same equation,
$$P_G(\phi_{\beta,F}) \geq \frac{1}{n_0}\log \mathbf{Z}_{n_0}(\phi_{\beta,F},a) + \frac{\log C}{n_0},$$
and by continuity of $\mathbf{Z}_{n_0}(\phi_{\beta,F},a)$ w.r.t $F$ , we deduce that $P_G(\phi_{\beta,F})\geq -2\epsilon$ for all $F$ small enough, which ends the proof.
\end{proof}

\paragraph{Characterization of the annealed free energy.} The starting point to prove Theorems \ref{critical_sum} and \ref{critical_exp} is the following characterization of the annealed free energy:
\begin{pr}\label{inf_caract}
 Suppose that $\sum_{n\geq1}n|\rho_n|$ is finite. For all $\beta\geq0$,
 \begin{enumerate}
  \item if $h\leq -(\beta^2/2) - P_G(\phi_{\beta})$ then $$F^a(\beta,h)=0,$$
  \item  if $h = -(\beta^2/2) - P_G(\phi_{\beta}) + \delta$ with $\delta>0$ then $F^a(\beta,h)$ is the unique positive solution of the following equation (w.r.t. $F$):
 \begin{equation}\label{inf_equ1}
P_G(\phi_{\beta,F}) = -\delta.
 \end{equation}
 \end{enumerate}
 \end{pr}
Note that hypothesis $\sum n|\rho_n|<+\infty$ ensures that the annealed free energy, $P_G(\phi_{\beta})$ and $P_G(\phi_{\beta,F})$ are well defined (see Proposition \ref{existence} and Remark \ref{rmk_existence}). An immediate consequence of Proposition \ref{inf_caract} is:
\begin{cor}\label{cor}For all $\beta\geq0$,
 \begin{equation*}
  h_c^a(\beta) = -(\beta^2/2) - P_G(\phi_{\beta}).
 \end{equation*}
\end{cor}

\begin{proof}[Proof of Proposition \ref{inf_caract}]
 Let us first point out that Equation (\ref{inf_equ1}) has indeed a unique positive solution because of Proposition \ref{fonction_P_G}. We slightly modify the expression of the annealed partition function by setting:
\begin{equation*}
 Z_{n,\beta,h}^{a} := E\left( \exp\left(\left(h+\frac{\beta^2}{2}\right)\sum_{k=1}^n \delta_k + \beta^2 \sum_{0\leq k < l \leq n} \rho_{l-k} \delta_k \delta_l\right) \delta_n\right),
\end{equation*}
(the second sum starts at $k=0$ instead of $k=1$) which does not affect the value of the annealed free energy.

\textit{Step 1: Adding a past.} In the following we call ``past'' a sequence of positive integers placed before $\tau_0 = 0$. We then get for any past $\overline{p}$ in $\Sigma$:
\begin{equation}\label{tronque}
 Z^{a}_{n,\beta,h} = \sum_{k=1}^n \displaystyle\sum_{\substack{l_1,\ldots,l_k \geq 1\\l_1 + \ldots + l_k = n}} \prod_{i=1}^k \left( e^{h+\frac{\beta^2}{2}+ \beta^2G^{(i)}(l_i l_{i-1}\ldots l_1 \overline{p})}K(l_i)\right).
\end{equation}
The inconvenient of this expression is that an infinite number of potential functions appear (the $G^{(i)}$'s, see Equation (\ref{truncated_G1})). Now, by analogy with the finite range correlations case (see \cite{0903.3704v3}), one would like to relate the annealed partition function $Z^{a}_{n,\beta,h}$ with the spectral properties of a unique operator (thus associated to a unique potential). The next step therefore consists in replacing the $G^{(i)}$'s by the potential $G$ (see Equation (\ref{def_potentialG})), that is replacing an infinite number of potentials with finite memory (which means that their value does not depend on the whole sequence but only on a finite set of index) by a unique potential with infinite memory. Since we assume that $\sum_{n\geq1}n|\rho_n|$ is finite and that the variations of $G$ are summable, the error produced by this operation tends to $0$ as $n$ goes to $+\infty$. For all past $\overline{p}$ in $\Sigma$, let us define: 
\begin{equation}\label{nontronque}
 \mathcal{Z}_{n,\beta,h}^{\overline{p}} := \sum_{k=1}^n \displaystyle\sum_{\substack{l_1,\ldots,l_k \geq 1\\l_1 + \ldots + l_k = n}} \prod_{i=1}^k \left( e^{h+\frac{\beta^2}{2}+ \beta^2G(l_i l_{i-1}\ldots l_1 \overline{p})}K(l_i)\right).
\end{equation}
Under the assumptions of the proposition, we have for all $\beta$ and $h$:
 \begin{equation}\label{limite_passe}
  F^a(\beta,h) = \lim_{n\rightarrow +\infty} (1/n) \log \mathcal{Z}_{n,\beta,h}^{\overline{p}}.
 \end{equation}
 Indeed, for all $i\geq1$, $l_1,\ldots,l_i\geq1$ and $\overline{p}$ in $\Sigma$, we have
 \begin{equation*}
  |G(l_i\ldots l_1\overline{p}) - G^{(i)}(l_i\ldots l_1\overline{p})| = \left|\sum_{k\geq0} \rho_{l_i + \ldots + l_1 + p_0 + \ldots + p_k} \right| \leq \sum_{k\geq i}|\rho_k|
 \end{equation*}
and from Equations (\ref{tronque}) and (\ref{nontronque}), we get the bounds
 \begin{equation}\label{encadr}
 e^{-\beta^2 \sum_{n\geq 1} n|\rho_n|}\mathcal{Z}_{n,\beta,h}^{\overline{p}}\leq Z^{a}_{n,\beta,h} \leq \mathcal{Z}_{n,\beta,h}^{\overline{p}} e^{\beta^2  \sum_{n\geq 1} n|\rho_n|},
 \end{equation}
from which it is straightforward to deduce Equation (\ref{limite_passe}). From now on, we can therefore work with $(\mathcal{Z}_{n,\beta,h}^{\overline{p}})_{n\geq1}$.

\textit{Step 2: Application of Theorem \ref{Ruelle_op_spec} and induced transition probabilities.} As $\sum_{n\geq1} n|\rho_n|$ is finite, the variations of $\phi_{\beta}$ (resp. $\phi_{\beta,F}$) are summable (Proposition \ref{lien_variations}), so we can apply Theorem \ref{Ruelle_op_spec}. Let us denote by $h_{\beta}$ and $\nu_{\beta}$ (resp. $h_{\beta,F}$ and $\nu_{\beta,F}$) the eigenfunction and eigenmeasure obtained thereby, normalized in such a way that $\int_{\Sigma} h_{\beta} d\nu_{\beta} = 1$ (resp. $\int_{\Sigma} h_{\beta,F} d\nu_{\beta,F} = 1$) and denote by $dm_{\beta} = h_{\beta} d\nu_{\beta}$ (resp. $dm_{\beta,F} = h_{\beta,F} d\nu_{\beta,F}$) the associated Gibbs measure. The induced transition probabilities, which were defined in Equation (\ref{chaine_induite}) in the general case, are written $(Q_{\beta}(\overline{s},\overline{t}))_{\overline{s},\overline{t} \in \Sigma}$ (resp. $(Q_{\beta,F}(\overline{s},\overline{t}))_{\overline{s},\overline{t} \in \Sigma}$), and the law of the Markov chain on $\Sigma$ associated to these transition probabilities, starting from $\overline{p}$, is denoted by $P_{\beta}(\cdot | \overline{p})$ (resp. $P_{\beta,F}(\cdot | \overline{p})$). One also defines
\begin{align*} 
 P_{\beta}(\cdot) &= \int_{\Sigma} P_{\beta}(\cdot | \overline{p}) dm_{\beta}(\overline{p})\\ \mbox{ (resp. } P_{\beta,F}(\cdot) &= \int_{\Sigma} P_{\beta,F}(\cdot | \overline{p}) dm_{\beta,F}(\overline{p})\mbox{)}.
 \end{align*} 
 Theses distributions can also be seen as distributions on the process $\tau$. Indeed, since $Q_{\beta}(x,y) >0$ if and only if $Ty = x$, one may define for all $x$ in $\Sigma$ and all $n\geq1$, 
\begin{equation}\label{Kbeta}
 K_{\beta}(x,n) := Q_{\beta}(x,nx) = \frac{\exp(\beta^2 G(nx))K(n)h_{\beta}(nx)}{\exp(P_G(\phi_{\beta}))h_{\beta}(x)}.
\end{equation}
Then, for all $x$ in $\Sigma$, $\sum_{n\geq 1} K_{\beta}(x,n) =1$, and for all $\overline{p}$ in $\Sigma$, $n\geq1$, and $l_1,\ldots,l_n \geq1$,
\begin{align}\label{Pbeta}
 &P_{\beta}(\tau_1 = l_1, \tau_2 = l_1 + l_2, \ldots, \tau_n = l_1 +\ldots + l_{n-1}+l_n | \overline{p}) \nonumber\\&= K_{\beta}(\overline{p},l_1)K_{\beta}(l_1\overline{p},l_2)\ldots K_{\beta}(l_{n-1}\ldots l_1 \overline{p},l_n)
\end{align}
Under $P_{\beta}(\cdot|\overline{p})$, the process $\tau = (\tau_n)_{n\geq0}$ has infinite range memory.\\
\textit{Step 3: Proof of item (1).} Let us consider the case $h\leq -(\beta^2/2) - P_G(\phi_{\beta})$ and show that $F^a(\beta,h) = 0$. Since free energy is nondecreasing in $h$, it is enough to prove that $$F^a(\beta, -\beta^2/2 - P_G(\phi_{\beta})) = 0.$$ For any past $\overline{p}$ we have
\begin{align*}
 &\mathcal{Z}_{n,\beta,-\beta^2/2 - P_G(\phi_{\beta})}^{\overline{p}} \\&= \sum_{k=1}^n \sum_{\substack{l_1,\ldots, l_k \geq 1\\ l_1 + \ldots + l_k = n}} \prod_{i=1}^k e^{\beta^2 G(l_i\ldots l_1 \overline{p}) - P_G(\phi_{\beta})}K(l_i)\\
 &  = \sum_{k=1}^n \sum_{\substack{l_1,\ldots, l_k \geq 1\\ l_1 + \ldots + l_k = n}} \frac{h_{\beta}(\overline{p})}{h_{\beta}(l_k\ldots l_1 \overline{p})}\prod_{i=1}^k K_{\beta}(l_{i-1}\ldots \overline{p},l_i) \mbox{\rm \, where $l_0\overline{p} := \overline{p}$\,} \\
 & \leq C_{\beta} \sum_{k=1}^n \sum_{\substack{l_1,\ldots, l_k \geq 1\\ l_1 + \ldots + l_k = n}} \prod_{i=1}^k K_{\beta}(l_{i-1}\ldots \overline{p},l_i)\mbox{\rm \, from item (5) of Theorem \ref{Ruelle_op_spec}}\\
 & = C_{\beta} P_{\beta}(n\in\tau | \overline{p})\\
 & \leq C_{\beta},
\end{align*}
then we use Equation (\ref{limite_passe}) to conclude.

\textit{Step 4: Proof of item (2).} Now we deal with the case $h = -\frac{\beta^2}{2} - P_G(\phi_{\beta}) + \delta$ when $\delta>0$. Let us denote by $F(\delta)$ the solution to Equation (\ref{inf_equ1}). By analogy with Equation (\ref{Kbeta}), we define for all $x$ in $\Sigma$ and $n\geq1$:
\begin{equation*}
 K_{\beta,F(\delta)}(x,n) = \exp(\delta + \beta^2 G(nx) - P_G(\phi_{\beta}) - F(\delta)n)K(n)\frac{h_{\beta,F(\delta)}(nx)}{h_{\beta,F(\delta)}(x)},
\end{equation*}
and for all past $\overline{p}$, $P_{\beta,F(\delta)}(\cdot|\overline{p})$ is the law on $\tau$ which is associated to the transition probabilities $K_{\beta,F(\delta)}$ (as in Equation (\ref{Pbeta})). Then we have:
\begin{align*}
 \mathcal{Z}_{n,\beta,h}^{\overline{p}} &= \sum_{k=1}^n \sum_{\substack{l_1,\ldots, l_k \geq 1\\ l_1 + \ldots + l_k = n}} \prod_{i=1}^k e^{\delta + \beta^2 G(l_i\ldots l_1 \overline{p}) - P_G(\phi_{\beta})}K(l_i)\\
 & = \exp(nF(\delta)) \sum_{k=1}^n \sum_{\substack{l_1,\ldots, l_k \geq 1\\ l_1 + \ldots + l_k = n}} \prod_{i=1}^k e^{\delta + \beta^2 G(l_i\ldots l_1 \overline{p}) - P_G(\phi_{\beta}) - F(\delta)l_i}K(l_i)\\
 & = \exp(nF(\delta)) \sum_{k=1}^n \sum_{\substack{l_1,\ldots, l_k \geq 1\\ l_1 + \ldots + l_k = n}} \frac{h_{\beta,F(\delta)}(\overline{p})}{h_{\beta,F(\delta)}(l_k\ldots l_1 \overline{p})} \prod_{i=1}^k K_{\beta,F(\delta)}(l_{i-1}\ldots \overline{p}, l_i).
\end{align*}
From item (5) of Theorem \ref{Ruelle_op_spec} and the fact that
\begin{equation*}
P_{\beta,F(\delta)}(n\in\tau | \overline{p}) = \sum_{k=1}^n \sum_{\substack{l_1,\ldots, l_k \geq 1\\ l_1 + \ldots + l_k = n}} \prod_{i=1}^k K_{\beta,F(\delta)}(l_{i-1}\ldots \overline{p}, l_i),
\end{equation*}
there exist two constants $0<c_{\beta}\leq C_{\beta}< +\infty$ such that
\begin{equation*}
 c_{\beta}P_{\beta,F(\delta)}(n\in\tau | \overline{p})\exp(F(\delta)n) \leq \mathcal{Z}_{n,\beta,h}^{\overline{p}} \leq C_{\beta}\exp(F(\delta)n).
\end{equation*}
Therefore, in order to finish the proof it is enough to show that
\begin{equation*}
(1/n)\log P_{\beta,F(\delta)}(n\in\tau | \overline{p}) \stackrel{n\rightarrow +\infty}{\longrightarrow} 0.
\end{equation*}
Now, fix $\beta>0$ and $\delta>0$, and let us write $P_{\overline{p}}$ instead of $P_{\beta,F(\delta)}(\cdot | \overline{p})$ to shorten notations. We show that there exists a constant $C>0$ such that for all $m,n\geq 1$,
 \begin{equation}\label{almost_subadd}
  P_{\overline{p}}(m+n\in\tau) \geq C P_{\overline{p}}(n\in\tau)P_{\overline{p}}(m\in\tau).
 \end{equation}
Multiplying both sides of the last equation by $C$, one can use a standard subadditivity argument to prove that the sequence $\left((1/n) \log P_{\overline{p}}(n\in\tau) \right)_{n\geq 1}$ converges to a limit which is clearly nonpositive. But if it were negative then we would get
\begin{equation*}
 E_{\overline{p}}\left(\sum_{n\geq 1} \delta_n\right) = \sum_{n\geq1} P_{\overline{p}}(n\in\tau) < +\infty,
\end{equation*}
which is absurd since $\sum_{n\geq1}\delta_n = +\infty$ almost surely. Thus the limit is $0$. Let us now prove Equation (\ref{almost_subadd}). Let $m,n\geq 1$. Then:
\begin{align*}
 &P_{\overline{p}}(n+m\in\tau) \\& \geq  P_{\overline{p}}(n\in\tau, n+m\in\tau) \\
 & = \sum_{\substack{1\leq k\leq n\\ 1\leq p\leq m}} \sum_{\substack{l_1+ \ldots + l_k = n\\j_1+\ldots+j_p = m}}  K_{\beta,F(\delta)}(\overline{p},l_1) K_{\beta,F(\delta)}(l_1\overline{p},l_2)\ldots K_{\beta,F(\delta)}(l_{n-1}\ldots l_1\overline{p},l_n)\\  &\hspace{3.5cm} \times K_{\beta,F(\delta)}(l_n\ldots l_1 \overline{p}, j_1)\ldots K_{\beta,F(\delta)}(j_{p-1} \ldots j_1 l_n\ldots l_1 \overline{p}, j_p).
 \end{align*}
But for all $i\geq 1$,
\begin{equation*}
 |G(j_i \ldots j_1 l_k \ldots l_1 \overline{p}) - G(j_i \ldots j_1 \overline{p})| \leq 2 \sum_{k\geq i} |\rho_k|.
\end{equation*}
Since we assumed $c = \sum_{i\geq 1} \sum_{k\geq i} |\rho_k| < +\infty$, we can set $$C = e^{-2c\beta^2 } \times \inf_{x,y\in\Sigma}\{h_{\beta,F(\delta)}(y)/h_{\beta,F(\delta)}(x)\},$$ which is a finite and positive constant (from item (5) of Theorem \ref{Ruelle_op_spec}), so we have:
\begin{align*}
 &P_{\overline{p}}(n+m\in\tau) \\& \geq  C \sum_{\substack{1\leq k\leq n\\ 1\leq p\leq m}} \sum_{\substack{l_1+ \ldots + l_k = n\\j_1+\ldots+j_p = m}} K_{\beta,F(\delta)}(\overline{p},l_1)K_{\beta,F(\delta)}(l_1\overline{p},l_2)\ldots K_{\beta,F(\delta)}(l_{n-1}\ldots l_1\overline{p},l_n)\\  &\hspace{3.5cm} \times K_{\beta,F(\delta)}(\overline{p}, j_1)\ldots K_{\beta,F(\delta)}(j_{p-1} \ldots j_1  \overline{p}, j_p)\\
 & \geq C P_{\overline{p}}(n\in\tau) P_{\overline{p}}(m\in\tau),
\end{align*}
which concludes the proof.
\end{proof}

\begin{proof}[Proof of Theorem \ref{critical_exp}]
 As we know from Proposition \ref{lien_variations}, if the $\rho_n$'s decrease exponentially fast then the same holds for the variations of the potentials $\phi_{\beta}$ and $\phi_{\beta,F}$, which are then H\"older continuous. Assume that $$h=h_c^a(\beta) + \delta = -(\beta^2/2) - P_G(\phi_{\beta}) + \delta$$ with $\delta>0$. From item (2) of Proposition \ref{inf_caract}, $F(\delta) := F^a(\beta,h)$ satisfies
\begin{equation}\label{inf_caract_equ}
 P_G(\phi_{\beta} - P_G(\phi_{\beta}) - F(\delta)\pi_0) = -\delta.
\end{equation}
We now apply Theorem \ref{Gurevich_pressure_critical_exp} with $\phi_{\beta} - P_G(\phi_{\beta})$ instead of $\phi$, and $-\pi_0$ instead of $\psi$, to obtain asymptotics for the left-hand term of Equation (\ref{inf_caract_equ}) when $\delta\searrow 0$, i.e $F(\delta)\searrow 0$. By ``inverting'' these asymptotics (see  \cite[Theorem 1.5.12]{MR1015093}), we get the critical behaviour of $F(\delta)$ as $\delta$ goes to $0$. 

In order to find the asymptotic behaviour of $m_{\beta}([n])$), which is needed to apply Theorem \ref{Gurevich_pressure_critical_exp}, we show that for all $x$ in $\Sigma$, the sequence $(h_{\beta}(nx))_{n\geq1}$ converges to a positive constant (not depending on $x$) that we denote by $\overline{h}_{\beta}$. Indeed, let us choose an arbitrary $a\geq 1$ and define for all $x$ in $\Sigma$ and $l\geq1$:
\begin{equation*}
h_{\beta}^{(l)}(x) := \exp(-lP_G(\phi_{\beta})) \left(\mathcal{L}^l_{\beta}\mathbf{1}_{[a]} \right)(x).
\end{equation*}
Let $m\geq n \geq 1$ be two integers. Then for all $l\geq1$, we have
\begin{equation*}
 e^{-2\beta^2 \sum_{j\geq n}\sum_{i\geq j}|\rho_i|}h_{\beta}^{(l)}(mx)\leq h_{\beta}^{(l)}(nx)\leq h_{\beta}^{(l)}(mx)  e^{2\beta^2 \sum_{j\geq n}\sum_{i\geq j}|\rho_i|}.
\end{equation*}
By making $l$ go to $+\infty$, using item (6) of Theorem \ref{Ruelle_op_spec} and taking the logarithm, one obtains:
\begin{equation*}
 \left| \log h_{\beta}(mx) - \log h_{\beta}(nx) \right| \leq 2\beta^2 \sum_{j\geq n} \sum_{i\geq j} |\rho_i| \stackrel{n \rightarrow + \infty}{\longrightarrow} 0.
\end{equation*}
The sequence $(\log h_{\beta}(nx))_{n\geq1}$ is therefore a Cauchy sequence. As a consequence, $(h_{\beta}(nx))_{n\geq1}$ converges to a finite positive limit. Moreover, it does not depend on $x$ since for all $x$ and $y$, we have
\begin{equation*}
 \left| \log h_{\beta}(nx) - \log h_{\beta}(ny) \right| \leq 2\beta^2 \sum_{j\geq n} \sum_{i\geq j} |\rho_i|.
\end{equation*}
By invariance of $m_{\beta}$ under $Q_{\beta}$, we have:
\begin{align}
 m_{\beta}([n]) &= \int Q_{\beta}(\overline{p},n\overline{p})dm_{\beta}(\overline{p})\nonumber\\&= \int e^{\beta^2G(n\overline{p}) - P_G(\phi_{\beta})} K(n) h_{\beta}(n\overline{p})d\nu_{\beta}(\overline{p})\label{mbeta_n},
\end{align}
which, by Dominated Convergence, gives
\begin{equation}\label{equ_mbeta}
 m_{\beta}([n]) \stackrel{n\rightarrow+\infty}{\sim} e^{-P_G(\phi_{\beta})}\overline{h}_{\beta}\nu_{\beta}(\Sigma) \times K(n) := c_{\beta} K(n).
\end{equation}
Using Theorem \ref{Gurevich_pressure_critical_exp} knowing Equations (\ref{equ_mbeta}) and (\ref{defK}), we get:\\
\begin{enumerate}
\item if $0<\alpha<1$, then $$\delta \sim_0 c'_{\beta} L(1/F(\delta))F(\delta)^{\alpha};$$
\item if $m =\sum_{n\geq 1} nK(n)<+\infty$ then $$\delta \sim_0 c'_{\beta} F(\delta);$$
\item if $\alpha=1$ and $\sum_{n\geq 1} nK(n)=+\infty$ then $$\delta \sim_0 c'_{\beta} \tilde{L}(1/F(\delta))F(\delta);$$
\end{enumerate}
where $c'_{\beta}$ is a positive constant (not the same from one case to another) and $\tilde{L}$ a slowly varying function. In order to get the asymptotic behaviour of $F(\delta)$ w.r.t $\delta$, we use \cite[Theorem 1.5.12]{MR1015093}.

Notice that when $$\sum_{n\geq1}nK(n)<+\infty,$$ one has
\begin{equation*}
 F(\delta) \stackrel{\delta\searrow 0}{\sim} \left( \int \pi_0 dm_{\beta} \right)^{-1} \delta,
\end{equation*}
and the constant $ \left( \int \pi_0 dm_{\beta} \right)^{-1}$ can be interpreted as a limit contact fraction when $\tau$ has distribution $m_{\beta}$.
\end{proof}

\begin{proof}[Proof of Theorem \ref{critical_sum}]
 We first deal with the case $\sum_{n\geq1}nK(n)<+\infty$. With the same techniques as in the proof of Proposition \ref{inf_caract}, and by considering \textit{free} rather than \textit{pinned} partition functions (see Remark \ref{pinned_free}), one can prove that for all $\delta>0$, if $h = h_c^a(\beta) + \delta = -(\beta^2/2) - P_G(\phi_{\beta})+\delta$, then
 \begin{equation*}
  \mathcal{Z}_{n,\beta,h}^{\overline{p}} \geq c_{\beta} E_{\beta}\left(\exp(\delta \imath_n)| \overline{p} \right)
 \end{equation*}
and by integrating on $\overline{p}$ and using Jensen's inequality:
\begin{equation}\label{borneinf}
 \frac{1}{n}\log\int  \mathcal{Z}_{n,\beta,h}^{\overline{p}}dm_{\beta}(\overline{p}) \geq \delta E_{\beta}(\imath_n/n) + o(1).
\end{equation}
Besides, from our assumptions, $\pi_0$ is $m_{\beta}$-integrable. Indeed,
\begin{align*}
\int \pi_0 dm_{\beta} &=  \sum_{n\geq1} n\times m_{\beta}(\{x : \pi_0(x) = n\})\\& = \sum_{n\geq1} n \int K(n)\frac{e^{\beta^2 G(n\overline{p})}}{e^{P_G(\beta)}}\frac{h_{\beta}(n\overline{p})}{h_{\beta}(\overline{p})}dm_{\beta}(\overline{p}) \mbox{\rm \, from Equation (\ref{mbeta_n})}\\&\leq C_{\beta} \sum_{n\geq 1} nK(n) < +\infty \mbox{\rm \, from item (5) of Theorem. \ref{Ruelle_op_spec}.}
\end{align*}
Since $m_{\beta}$ is ergodic, Birkhoff's Theorem tells us that 
\begin{equation*}
\frac{\tau_n}{n} = \frac{1}{n}\sum_{i=1}^n T_i \stackrel{n\rightarrow +\infty}{\longrightarrow} \int \pi_0 dm_{\beta} \quad \mbox{\rm $m_{\beta}$-a.s.}
\end{equation*}
From the bounds
\begin{equation*}
 \tau_{\imath_n} \leq n <  \tau_{\imath_n + 1},
\end{equation*}
we get that $m_{\beta}$-a.s (and in $L^1(m_{\beta})$ by Dominated Convergence),
\begin{equation}\label{cvg_mbeta}
 \frac{\imath_n}{n} \stackrel{n\rightarrow + \infty}{\longrightarrow} \left(\int\pi_0 dm_{\beta}\right)^{-1} > 0.
\end{equation}
Notice that the convergence in Equation (\ref{limite_passe}) remains true if we replace $\mathcal{Z}_n^{\overline{p}}$ by $\int \mathcal{Z}_n^{\overline{p}}dm_{\beta}(\overline{p})$, from the bounds of Equation (\ref{encadr}). By taking the large $n$ limit  in Equation (\ref{borneinf}) and using Equation (\ref{cvg_mbeta}), we get
\begin{equation*}
 F^a(\beta,h_c^a(\beta) + \delta) \geq \left(\int \pi_0 dm_{\beta} \right)^{-1}\delta.
\end{equation*}
The upper bound in Equation (\ref{critical_sum_moy}) comes from the following inequality:
$$(h_c^a(\beta) + \delta) \imath_n \leq h_c^a(\beta)\imath_n + \delta n.$$

We now come back to \textit{pinned} partition functions and deal with the other cases. In what follows, the parameter $\beta>0$ is kept fixed. The idea of the proof is to get bounds on $P_G(\phi_{\beta,F})$ for small values of $F$ by applying the operator $\mathcal{L}_{\beta,F}$ to the function $h_{\beta}/(h_{\beta}\circ T)$, where $h_{\beta}$ is the eigenfunction of the operator $\mathcal{L}_{\beta}$ associated to the eigenvalue $\exp(P_G(\phi_{\beta}))$, which is given by Theorem \ref{Ruelle_op_spec}. Indeed, we have for all $x$ in $\Sigma$,
\begin{align*}
 \left(\mathcal{L}_{\beta,F}\frac{h_{\beta}}{h_{\beta}\circ T}\right)(x) &= \sum_{n\geq1} e^{\beta^2 G(nx) - P_G(\phi_{\beta})}\frac{h_{\beta}(nx)}{h_{\beta}(x)}K(n)e^{-nF}\\
 & = \sum_{n\geq1} e^{\beta^2 G(nx) - P_G(\phi_{\beta})}\frac{h_{\beta}(nx)}{h_{\beta}(x)}K(n) \\&\hspace{1cm}+ \sum_{n\geq1} e^{\beta^2 G(nx) - P_G(\phi_{\beta})}\frac{h_{\beta}(nx)}{h_{\beta}(x)}K(n)(e^{-nF}-1).
\end{align*}
The sum in the second line is equal to $1$ by definition of $h_{\beta}$, whereas by boundedness of $G$ and item (5) of Theorem \ref{Ruelle_op_spec}, the sum in the last line is bounded above and below by a constant (which does not depend on $x$) times
\begin{equation*}
 \vartheta(F) : = \sum_{n\geq1} K(n)(e^{-nF}-1).
\end{equation*}
We then prove by induction that for small enough values of $F$, there exists a constant $C$ such that for all $n\geq1$ and for all $x$ in $\Sigma$,  we have
\begin{align}
 \left(\mathcal{L}^n_{\beta,F}\frac{h_{\beta}}{h_{\beta}\circ T^n} \right)(x) &\leq (1 + C \vartheta(F))^n \label{rec_1}\\
 \mbox{\rm  and} \quad \left(\mathcal{L}^n_{\beta,F}\frac{h_{\beta}}{h_{\beta}\circ T^n} \right)(x) &\geq (1 - (1/C) \vartheta(F))^n,\label{rec_2}
\end{align}
which is what we have just proved above for $n=1$ and all $x$. Suppose Equations (\ref{rec_1}) and (\ref{rec_2}) hold for some $n\geq1$ and all $x$. We prove that both relations hold for $n+1$. Let $x\in\Sigma$. Recall Equations (\ref{phi_n}) and (\ref{iter_operator_fct}). We have
\begin{align*}
 \left(\mathcal{L}^{n+1}_{\beta,F}\frac{h_{\beta}}{h_{\beta}\circ T^{n+1}} \right)(x) &= \sum_{k_1,\ldots,k_{n+1}\geq1}e^{\phi_{\beta,F,n+1}(k_{n+1}\ldots k_1 x)}\frac{h_{\beta}(k_{n+1}\ldots k_1 x)}{h_{\beta}(x)}\\
 & = \sum_{k_1,\ldots,k_n\geq1}e^{\phi_{\beta,F,n}(k_n \ldots k_1x)}\frac{h_{\beta}(k_n \ldots k_1 x)}{h_{\beta}(x)}\\ &\hspace{5mm}\times\sum_{k_{n+1}\geq 1} e^{\phi_{\beta,F}(k_{n+1}k_n \ldots k_1 x)}\frac{h_{\beta}(k_{n+1}k_n \ldots k_1 x)}{h_{\beta}(k_n \ldots k_1 x)}
\end{align*}
Using Equations (\ref{rec_1}) and (\ref{rec_2}) for $n=1$ and $x$ replaced by $(k_n\ldots k_1 x)$, one gets
\begin{align*}
 \left(\mathcal{L}^{n+1}_{\beta,F}\frac{h_{\beta}}{h_{\beta}\circ T^{n+1}} \right)(x) &\leq (1 + C\vartheta(F)) \sum_{k_1,\ldots,k_n\geq1}e^{\phi_{\beta,F,n}(k_n \ldots k_1x)}\frac{h_{\beta}(k_n \ldots k_1 x)}{h_{\beta}(x)}\\
 & = (1 + C\vartheta(F)) \left(\mathcal{L}^n_{\beta,F}\frac{h_{\beta}}{h_{\beta}\circ T^n} \right)(x)\\
 &\leq (1 + C\vartheta(F))^{n+1} \mbox{ by induction hypothesis.}
\end{align*}
We prove in the same way that 
\begin{equation*}
\left(\mathcal{L}^{n+1}_{\beta,F}\frac{h_{\beta}}{h_{\beta}\circ T^{n+1}} \right)(x) \geq (1 - (1/C)\vartheta(F))^{n+1},
\end{equation*}
which ends the induction. Combining Equations (\ref{rec_1}) and (\ref{rec_2}) with Equation (\ref{alter_def2}) (recall that $0<\inf h_{\beta} \leq \sup h_{\beta} < + \infty$), we get for small values of $F$:
\begin{equation}\label{ingr1}
 - (1/C) \vartheta(F) \leq P_G(\phi_{\beta,F}) \leq - C \vartheta(F).
\end{equation}
Since $K(n) = L(n)n^{-(1+\alpha)}$, we get (see \cite[Corollary 8.1.7]{MR1015093})
\begin{equation}\label{ingr2}
 \vartheta(F) \sim cL(1/F)F^{\alpha} 
\end{equation}
for some constant $c$. Combining Equations (\ref{ingr1}) and (\ref{ingr2}) applied to $F = F^a(\beta,h_c^a(\beta)+\delta)$ for small positive values of $\delta$, with item (2) of Proposition \ref{inf_caract} and \cite[Theorem 1.5.12]{MR1015093}, we get the result.
\end{proof}

\section{Proof of Theorem \ref{asympt}}
\begin{proof} 
Recall Equation (\ref{def_potentialG}). For every sequence of positive integers $\overline{t}=(t_i)_{i\geq 0}$ and for all $q\geq 1$, let us define 
\begin{equation}\label{truncated_G}
 G^{[q]}(\overline{t}) := \sum_{n\geq0} \rho_{t_0 + \ldots + t_n} \mathbf{1}_{\{t_0 + \ldots + t_n \leq q\}}.
\end{equation}
Since
\begin{equation*}
 \left| G^{[q]}(\overline{t}) - G(\overline{t}) \right| \leq  \sum_{k\geq q+1} |\rho_k|,
\end{equation*}
we have
\begin{equation*}
F^{a,q}\left(\beta, h - \beta^2 \sum_{k\geq q+1} |\rho_k| \right) \leq F^a(\beta, h) \leq F^{a,q}\left(\beta, h + \beta^2 \sum_{k\geq q+1} |\rho_k| \right),
\end{equation*}
where $F^{a,q}$ (resp. $h_c^{a,q}$) is the free energy (resp. critical curve) of the annealed model with $\rho_n$ replaced by $\rho_n\mathbf{1}_{\{n\leq q\}}$. Hence, we have
\begin{equation*}
 h_c^{a,q}(\beta) - \beta^2 \sum_{k\geq q+1}|\rho_k| \leq h_c^a(\beta) \leq h_c^{a,q}(\beta) + \beta^2 \sum_{k\geq q+1}|\rho_k|
\end{equation*}
from which we deduce using \cite[Proposition 4.2]{0903.3704v3},
\begin{align*}
 \limsup_{\beta\searrow 0} \frac{h_c^a(\beta)}{(\beta^2/2)} &\leq -\left(1+ 2\sum_{k=1}^q \rho_k P(k\in\tau) - 2 \sum_{k\geq q+1}|\rho_k|\right),\\
  \liminf_{\beta\searrow 0} \frac{h_c^a(\beta)}{(\beta^2/2)} &\geq -\left(1+ 2\sum_{k=1}^q \rho_k P(k\in\tau) + 2\sum_{k\geq q+1}|\rho_k|\right).
\end{align*}
We conclude by making $q$ tend to $+\infty$.
\end{proof}

If $m = \sum_{n\geq1}nK(n) <+\infty$, we can also prove the upper bound by applying Jensen's inequality on the annealed partition function. We get indeed by considering \textit{free} partition functions (see Remark \ref{pinned_free}):
\begin{equation*}
 Z^a_{n,\beta,h} \geq \exp\left(\left(h + \frac{\beta^2}{2} \right)\sum_{k=1}^n P(k\in\tau) + \beta^2 \sum_{1\leq k < l \leq n} \rho_{l-k}P(k\in\tau)P(l-k\in\tau)\right),
\end{equation*}
which gives
\begin{equation*}
 \frac{1}{n}\log  Z^a_{n,\beta,h} \geq \left(h + \frac{\beta^2}{2} \right)\frac{1}{n}\sum_{k=1}^n P(k\in\tau) + \frac{\beta^2}{n}\sum_{k=1}^{n-1}P(k\in\tau)\sum_{l=1}^{n-k} \rho_l P(l\in\tau).
\end{equation*}
Take the large $n$ limit and use the Renewal Theorem to get:
\begin{equation*}
 F^a(\beta,h) \geq \frac{1}{m}\left(h + \frac{\beta^2}{2}\left(1 + 2\sum_{n\geq 1}\rho_n P(n\in\tau) \right) \right),
\end{equation*}
and so, for all $\beta\geq0$, for all $h > -\frac{\beta^2}{2}\left(1 + 2\sum_{n\geq 1}\rho_n P(n\in\tau) \right)$, we have $F^a(\beta,h)>0$, which means that:
\begin{equation*}
 \forall \beta\geq0,\quad h_c^a(\beta) \leq -\frac{\beta^2}{2}\left(1 + 2\sum_{n\geq 1}\rho_n P(n\in\tau) \right).
\end{equation*}
Note that in this case the upper bound holds for all $\beta$.

\bigskip

{\bf Acknowledgements.} The author is grateful to an anonymous referee for valuable comments and improvement on Proposition \ref{existence}.

\bibliographystyle{spmpsci.bst}      
\bibliography{references.bib}
\end{document}